\documentclass[a4 paper, reqno]{amsart}
\usepackage{latexsym}
\usepackage
[
margin=3cm
]
{geometry}
\newcommand{\be}{\begin{equation}}
\newcommand{\ee}{\end{equation}}
\newcommand{\bea}{\begin{eqnarray}}
\newcommand{\eea}{\end{eqnarray}}
\newcommand{\bean}{\begin{eqnarray*}}
\newcommand{\eean}{\end{eqnarray*}}

\newcommand{\brray}{\begin{array}}
\newcommand{\erray}{\end{array}}
\newcommand{\ben}{\begin{equation}{nonumber}}
\newcommand{\een}{\end{equation}{nonumber}}

\newtheorem{defn}{Definition}[section]
\newtheorem{thm}[defn]{Theorem}
\newtheorem{lemma}[defn]{Lemma}
\newtheorem{prop}[defn]{Proposition}
\newtheorem{corr}[defn]{Corollary}
\newtheorem{xmpl}[defn]{Example}
\newtheorem{rmk}[defn]{Remark}
\newcommand{\bdfn}{\begin{defn}}
\newcommand{\bthm}{\begin{thm}}
\newcommand{\blmma}{\begin{lemma}}
\newcommand{\bppsn}{\begin{prop}}
\newcommand{\bcrlre}{\begin{corr}}
\newcommand{\bxmpl}{\begin{xmpl}}
\newcommand{\brmrk}{\begin{rmk}}
\newcommand{\edfn}{\end{defn}}
\newcommand{\ethm}{\end{thm}}
\newcommand{\elmma}{\end{lemma}}
\newcommand{\eppsn}{\end{prop}}
\newcommand{\ecrlre}{\end{corr}}
\newcommand{\exmpl}{\end{xmpl}}
\newcommand{\ermrk}{\end{rmk}}

\newcommand{\IC}{\mathbb{C}}

\newcommand{\id}{\mathrm{id}}

\newcommand{\twoform}{{{\Omega}^2}( \mathcal{A} )}
\newcommand{\tensora}{\otimes_{\mathcal{A}}}

\newcommand{\tensorc}{\otimes_{\mathbb{C}}}
\newcommand{\A}{\mathcal{A}}

\newcommand{\E}{\mathcal{E}}

\newcommand{\zeroE}{{}_0\E}
\newcommand{\Ezero}{\E_0}

\newcommand{\Hom}{{\rm Hom}}

\newcommand{\deltam}{\Delta_M}

\newcommand{\mdelta}{{}_M \Delta}

\newcommand{\zeroM}{{}_0 M}
\newcommand{\Mzero}{M_0}

\newcommand{\bimodbicov}{{}^\mathcal{A}_\mathcal{A} \mathcal{M}^\mathcal{A}_\mathcal{A}}
\newcommand{\bimodrightcov}{{}_\mathcal{A} \mathcal{M}^\mathcal{A}_\mathcal{A}}
\newcommand{\bimodleftcov}{{}^\mathcal{A}_\mathcal{A} \mathcal{M}_\mathcal{A}}
\newcommand{\rightmodbicov}{{}^\mathcal{A} \mathcal{M}^\mathcal{A}_\mathcal{A}}
\newcommand{\leftcov}{{}^\mathcal{A} \mathcal{M}}
\newcommand{\bimodbicovgamma}{{}^{\mathcal{A}_\gamma}_{\mathcal{A}_\gamma} \mathcal{M}^{\mathcal{A}_\gamma}_{\mathcal{A}_\gamma}}

\newcommand{\leftcovgamma}{{}^{\mathcal{A}_\gamma} \mathcal{M}}
\newcommand{\ev}{\rm ev}
\newcommand{\coev}{\rm coev}

\newcommand{\RNum}[1]{\uppercase\expandafter{\romannumeral #1\relax}}

\begin{document}
\title{Pseudo-Riemannian metrics on bicovariant bimodules}
\maketitle
\begin{center}
{\large {Jyotishman Bhowmick and Sugato Mukhopadhyay}}\\
Indian Statistical Institute\\
203, B. T. Road, Kolkata 700108\\
Emails: jyotishmanb$@$gmail.com, m.xugato@gmail.com \\
\end{center}
\begin{abstract}
We study pseudo-Riemannian invariant metrics on bicovariant bimodules over Hopf algebras. We clarify some properties of such metrics and prove that pseudo-Riemannian invariant metrics on a bicovariant bimodule and its cocycle deformations are in one to one correspondence.
\end{abstract}

\section{Introduction}	

The notion of metrics on covariant bimodules on Hopf algebras have been studied by a number of authors including Heckenberger and~ Schm{\"u}dgen (\cite{heckenberger}, \cite{heckenbergerlaplace}, \cite{heckenbergerspin}) as well as Beggs, Majid and their collaborators (\cite{beggsmajidbook} and references therein). The goal of this article is to characterize bicovariant pseudo-Riemannian metrics on a cocycle-twisted bicovariant bimodule. As in \cite{heckenberger}, the symmetry of the metric comes from Woronowicz's braiding map $ \sigma $ on the bicovariant bimodule. However, since our notion of non-degeneracy of the metric is slightly weaker than that in \cite{heckenberger}, we consider a slightly larger class  of metrics than those in \cite{heckenberger}. The positive-definiteness of the metric does not play any role in what we do.

We refer to the later sections for the definitions of pseudo-Riemannian metrics and cocycle deformations. Our strategy is to exploit the covariance of the various maps between bicovariant bimodules to view them as maps between the finite-dimensional vector spaces of left-invariant elements of the respective bimodules. This was already observed and used crucially by Heckenberger and~ Schm{\"u}dgen in the paper \cite{heckenberger}. We prove that bi-invariant pseudo-Riemannian metrics are automatically bicovariant maps and compare our definition of pseudo-Riemannian metric with some of the other definitions available in the literature. Finally, we prove that the pseudo-Riemannian bi-invariant metrics on a bicovariant bimodule and its cocycle deformation are in one to one correspondence. These results will be used in the companion article \cite{article6}.

In Section \ref{section2}, we discuss some generalities on bicovariant bimodules. In Section \ref{section3}, we define and study pseudo-Riemannian left metrics on a bicovariant differential calculus.  Finally, in Section \ref{21staugust20197}, we prove our main result on bi-invariant metrics on cocycle-deformations.

Let us set up some notations and conventions that we are going to follow. All vector spaces will be assumed to be over the complex field. For vector spaces $ V_1 $ and $ V_2, $ $ \sigma^{{\rm can}} : V_1 \tensorc V_2 \rightarrow V_2 \tensorc V_1 $ will denote the canonical flip map, i.e, $ \sigma^{{\rm can}} (v_1 \tensorc v_2) = v_2 \tensorc v_1. $ For the rest of the article, $(\A,\Delta)$ will denote a Hopf algebra. We will use the Sweedler notation for the coproduct $\Delta$. Thus, we will write
\begin{equation} \label{28thaugust20191} 
\Delta(a) = a_{(1)} \tensorc a_{(2)}.
\end{equation}
For a right $\A$-module $V,$ the notation $V^*$ will stand for the set $ \Hom_{\A} ( V, \A ). $

Following \cite{woronowicz}, the comodule coaction on a left $\A$-comodule $V$ will be denoted by the symbol $ \Delta_V. $ Thus, $ \Delta_V $ is a $\IC$-linear map $\Delta_V:V \to \A \tensorc V$ such that
$$ (\Delta \tensorc \id) \Delta_V = (\id \tensorc \Delta_V) \Delta_V, ~ (\epsilon \tensorc \id)\Delta_V(v)=v $$
for all $ v $ in $ V $ (here $\epsilon$ is the counit of $\A$). We will use the notation
\begin{equation} \label{28thaugust20192}
\Delta_V(v) = v_{(-1)} \tensorc v_{(0)}.
\end{equation}
Similarly, the comodule coaction on a right $\A$-comodule will be denoted by $ {}_V \Delta $ and we will write 
\begin{equation} \label{28thaugust20193}
{}_V \Delta(v) = v_{(0)} \tensorc v_{(1)}.
\end{equation}	

Finally, for a Hopf algebra $\A,$ $ \bimodleftcov, \bimodrightcov, \bimodbicov $ will denote the categories of various types of mixed Hopf-bimodules as in Subsection 1.9 of \cite{montgomery}.

\section{Covariant bimodules on quantum groups} \label{section2}

In this section we recall and prove some basic facts on covariant bimodules. These objects were studied by many Hopf-algebraists (as Hopf-bimodules) including Abe (\cite{abe}) and Sweedler (\cite{sweedler}). During the 1980's, they were re-introduced by Woronowicz (\cite{woronowicz}) for studying differential calculi over Hopf algebras. Schauenburg (\cite{schauenberg}) proved a categorical equivalence between bicovariant bimodules and Yetter-Drinfeld modules over a Hopf algebra, the latter being introduced by Yetter in \cite{yetter}.

We start by recalling the notions on covariant bimodules from Section 2 of \cite{woronowicz}. Suppose $M$ is a bimodule over $\A$ such that $ (M, \Delta_M) $ is a left $\A$-comodule. Then $ (M, \Delta_M) $ is called a left-covariant bimodule if this tuplet is an object of the category $ \bimodleftcov, $ i.e, for all $a$ in $\A$ and $m$ in $M$, the following equation holds.
$$\Delta_M(a m)=\Delta(a)\Delta_M(m),~ \Delta_M(m a)=\Delta_M(m)\Delta(a).$$
Similarly, if $ {}_M \Delta $ is a right comodule coaction on $M,$ then $ (M, {}_M \Delta) $ is called a right covariant bimodule if it is an object of the category $ \bimodrightcov, $ i.e, for any $a$ in $\A$ and $m$ in $M$, 
$${}_ M\Delta(a m)=\Delta(a){}_ M\Delta(m),~ {}_ M\Delta(m a)={}_ M\Delta(m)\Delta(a).$$
Finally, let $ M$ be a bimodule over $\A$ and $\Delta_ M: M \to \A \tensorc M$ and ${}_ M \Delta: M \to M \tensorc \A$ be $\IC$-linear maps. Then we say that $(M, \Delta_{ M}, {}_ M \Delta)$ is a bicovariant bimodule if this triplet is an object of $\bimodbicov. $ Thus, 
\begin{itemize}
	\item[(i)] $(M, \Delta_{ M})$ is left-covariant bimodule,
	\item[(ii)] $(M, {}_ M \Delta)$ is a right-covariant bimodule,
	\item[(iii)] $ (\id \tensorc {}_M \Delta) \Delta_M = (\Delta_M \tensorc \id) {}_M\Delta. $
\end{itemize}

The vector space of left (respectively, right) invariant elements of a left (respectively, right) covariant bimodules will play a crucial role in the sequel and we introduce notations for them here.
\begin{defn} \label{21staugust20193}
	Let $(M, \Delta_M)$ be a left-covariant bimodule over $\A$. The subspace of left-invariant elements of $M$ is defined to be the vector space 
	$${}_0M:=\{m \in M : \Delta_M(m)=1\tensorc m\}.$$
	Similarly, if $(M, {}_M\Delta)$ is a right-covariant bimodule over $\A$ , the subspace of right-invariant elements of $M$ is the vector space 
	$$M_0:=\{m \in M : {}_M\Delta(m) = m \tensorc 1\}.$$
\end{defn}

\brmrk \label{20thjune}
We will say that a bicovariant bimodule $( M, \Delta_M, {}_M \Delta ) $ is finite if $ {}_0 M $ is a finite dimensional vector space. Throughout this article, we will only work with bicovariant bimodules which are finite. 
\ermrk

Let us note the immediate consequences of the above definitions.
\blmma \label{1staugust2019jb1} (Theorem 2.4, \cite{woronowicz} )
Suppose $ M $ is a bicovariant $\A$-$\A$-bimodule. Then
$$ \mdelta (\zeroM) \subseteq \zeroM \tensorc \A,~ \deltam (\Mzero) \subseteq \A \tensorc \Mzero. $$
More precisely, if $ \{ m_i \}_i $ is a basis of $ \zeroM, $ then there exist elements $ \{ a_{ji} \}_{i,j} $ in $ \A $ such that
\begin{equation} \label{26thaugust20191} \mdelta (m_i) = \sum_j m_j \tensorc a_{ji}. \end{equation}
\elmma
\begin{proof} This is a simple consequence of the fact that $ \mdelta $ commutes with $ \deltam. $ 
\end{proof}

The category $\bimodleftcov$ has a natural monoidal structure. Indeed, if
$ (M, \Delta_M) $ and $ (N, \Delta_N) $ are left-covariant bimodules over $ \A, $ then we have a left coaction $ \Delta_{M \tensora N} $ of $ \A $ on $ M \tensora N $ defined by the following formula:
$$ \Delta_{M \tensora N} (m \tensora n) = m_{(-1)} n_{(-1)} \tensorc m_{(0)} \tensora n_{(0)}. $$
Here, we have made use of the Sweedler notation introduced in \eqref{28thaugust20192}. This makes $ M \tensora N $ into a left covariant $\A$-$\A$-bimodule. Similarly, there is a right coaction $ {}_{M \tensora N} \Delta $ on $ M \tensora N $ if $ (M, {}_M \Delta) $ and $ (N, {}_N \Delta)$ are right covariant bimodules. 

The fundamental theorem of Hopf modules (Theorem 1.9.4 of \cite{montgomery}) states that if $ V $ is a left-covariant bimodule over $\A,$ then $ V $ is a free as a left (as well as a right) $ \A$-module. This was reproved by Woronowicz in \cite{woronowicz}. In fact, one has the following result:

\begin{prop}{(Theorem 2.1 and Theorem 2.3 of \cite{woronowicz})} \label{moduleiso}
	Let $(M, \Delta_M)$ be a bicovariant bimodule over $\A$. Then the multiplication maps ${}_0M \tensorc \A \to M,$ $ \A \tensorc {}_0M \to M,$ $ M_0 \tensorc \A \to M$ and $ \A \tensorc M_0 \to M$ are isomorphisms.
\end{prop}

\begin{corr} \label{3rdaugust20191}
	Let $(M,\Delta_M)$ and $(N, \Delta_N)$ be left-covariant bimodules over $\A$ and $\{m_i\}_i$ and $\{n_j\}_j$ be vector space bases of ${}_0 M$ and ${}_0 N$ respectively. Then each element of $M \tensora N$ can be written as $\sum_{ij} a_{ij} m_i \tensora n_j$ and $\sum_{ij} m_i \tensora n_j b_{ij}$, where $a_{ij}$ and $b_{ij}$ are uniquely determined.\\
	A similar result holds for right-covariant bimodules $(M, {}_M \Delta)$ and $(N, {}_N \Delta)$ over $\A$. Finally, if $(M,\Delta_M)$ is a left-covariant bimodule over $\A$ with basis $\{m_i\}_i$ of ${}_0 M$, and $(N, {}_N \Delta)$ is a right-covariant bimodule over $\A$ with basis $\{n_j\}_j$ of $N_0$, then any element of $M \tensora N$ can be written uniquely as $\sum_{ij} a_{ij} m_i \tensora n_j$ as well as $\sum_{ij} m_i \tensora n_j b_{ij}$.
\end{corr}
\begin{proof}
	The proof of this result is an adaptation of Lemma 3.2 of \cite{woronowicz} and we omit it.
\end{proof}	

The next proposition will require the definition of right Yetter-Drinfeld modules for which we refer to \cite{yetter} and Definition 4.1 of \cite{schauenberg}.

\begin{prop} \label{3rdaugust20192} (Theorem 5.7 of \cite{schauenberg})
	The functor $ M \mapsto {}_0 M $ induces a monoidal equivalence of categories $ \bimodbicov $ and the category of right Yetter-Drinfeld modules. Therefore, if	 $(M,\Delta_M)$ and $(N, \Delta_M)$ be left-covariant bimodules over $\A,$ then
	\begin{equation} \label{21staugust20194}
	{}_0 (M \tensora N) = {\rm span}_\IC \{m \tensora n : m \in {}_0 M, n \in {}_0 N \}.
	\end{equation}
	Similarly, if $(M, {}_M \Delta)$ and $(N, {}_N \Delta)$ are right-covariant bimodules over $\A$, then we have that $$ (M \tensora N)_0 = {\rm span}_\IC \{m \tensora n : m \in M_0, n \in N_0 \}.$$
	Thus, ${}_0 (M \tensora N) = {}_0 M \tensorc {}_0 N$ and $(M \tensora N)_0 = M_0 \tensorc N_0$.
\end{prop}

\begin{rmk} \label{29thjune20191}
	In the light of Proposition \ref{3rdaugust20192}, we are allowed to use the notations $ {}_{0} M \tensorc {}_{0} N $ and $ {}_{0} (M \tensora N) $ interchangeably.
\end{rmk}

We recall now the definition of covariant maps between bimodules.

\begin{defn}
	Let $(M, \Delta_{ M}, {}_M \Delta)$ and $(N, \Delta_N, {}_N \Delta)$ be bicovariant $\A$-bimodules and $ T $ be a $ \mathbb{C} $-linear map from $ M $ to $ N. $
	
	$ T $ is called left-covariant if $T$ is a morphism in the category $\leftcov,$ i.e, for all $ m \in M, $
	$$ (\id \tensorc T)(\Delta_{M}(m))=\Delta_N(T(m)). $$
	$T$ is called right-covariant if $T$ is a morphism in the category $\mathcal{M}^\A.$ Thus, for all $ m \in M,$
	$$ ( T \tensorc \id) {}_M \Delta (m) = {}_N \Delta (T (m)). $$
	Finally, a map which is both left and right covariant will be called a bicovariant map. In other words, a bicovariant map is a morphism in the category ${}^\A \mathcal{M}^\A.$
\end{defn}

We end this section by recalling the following fundamental result of Woronowicz.

\begin{prop}{(Proposition 3.1 of \cite{woronowicz})} \label{4thmay20193}
	Given a bicovariant bimodule $\E$ there exists a unique bimodule homomorphism
	$$\sigma: \E \tensora \E \to \E \tensora \E ~ {\rm such} ~ {\rm that} $$ 
	\be \label{30thapril20191} \sigma(\omega \tensora \eta)= \eta \tensora \omega \ee
	for any left-invariant element $\omega$ and right-invariant element $\eta$ in $\E$. $\sigma$ is invertible and is a bicovariant $\A$-bimodule map from $\E \tensora \E$ to itself. Moreover, $\sigma$ satisfies the following braid equation on $\E \tensora \E \tensora \E:$
	$$ (\id \tensora \sigma)(\sigma \tensora \id)(\id \tensora \sigma)= (\sigma \tensora \id)(\id \tensora \sigma)(\sigma \tensora \id). $$
\end{prop}

\section{Pseudo-Riemannian metrics on bicovariant bimodules} \label{section3}

In this section, we study pseudo-Riemannian metrics on bicovariant differential calculus on Hopf algebras. After defining pseudo-Riemannian metrics, we recall the definitions of left and right invariance of a pseudo-Riemannian metrics from \cite{heckenberger}. We prove that a pseudo-Riemannian metric is left (respectively, right) invariant if and only if it is left (respectively, right) covariant. The coefficients of a left-invariant pseudo-Riemannian metric with respect to a left-invariant basis of $\E$ are scalars. We use this fact to clarify some properties of pseudo-Riemannian invariant metrics. We end the section by comparing our definition with those by Heckenberger and~ Schm{\"u}dgen (\cite{heckenberger}) as well as by Beggs and Majid.

\begin{defn} \label{24thmay20191} (\cite{heckenberger})
	Suppose $ \E $ is a bicovariant $\A$-bimodule $\E$ and $ \sigma: \E \tensora \E \rightarrow \E \tensora \E $ be the map as in Proposition \ref{4thmay20193}. A pseudo-Riemannian metric for the pair $ (\E, \sigma) $ is a right $\A$-linear map $g:\E \tensora \E \to \A$ such that the following conditions hold:
	\begin{enumerate}
		\item[(i)] $ g \circ \sigma = g. $
		\item[(ii)] If $g(\rho \tensora \nu)=0$ for all $\nu$ in $\E,$ then $\rho = 0.$ 
	\end{enumerate}
\end{defn}

For other notions of metrics on covariant differential calculus, we refer to \cite{beggsmajidbook} and references therein.

\begin{defn} (\cite{heckenberger})
	A pseudo-Riemannian metric $g$ on a bicovariant $\A$-bimodule $\E$ is said to be left-invariant if for all $\rho, \nu$ in $\E$,
	$$ (\id \tensorc \epsilon g)(\Delta_{(\E \tensora \E)}(\rho \tensora \nu)) = g(\rho \tensora \nu). $$
	Similarly, a pseudo-Riemannian metric $g$ on a bicovariant $\A$-bimodule $\E$ is said to be right-invariant if for all $\rho, \nu$ in $\E$,
	$$ (\epsilon g \tensorc \id)({}_{(\E \tensora \E)}\Delta(\rho \tensora \nu)) = g(\rho \tensora \nu). $$
	Finally, a pseudo-Riemannian metric $g$ on a bicovariant $\A$-$\A$ bimodule $\E$ is said to be bi-invariant if it is both left-invariant as well as right-invariant.
\end{defn}

We observe that a pseudo-Riemannian metric is invariant if and only if it is covariant.

\begin{prop} \label{29thaugust20192}
	Let $g$ be a pseudo-Riemannian metric on the bicovariant bimodule $\E$. Then $g$ is left-invariant if and only if $g: \E \tensora \E \to \A$ is a left-covariant map. Similarly, $g$ is right-invariant if and only if $g: \E \tensora \E \to \A$ is a right-covariant map.
\end{prop}
\begin{proof}
	Let $g$ be a left-invariant metric on $\E$, and $\rho$, $\nu$ be elements of $\E$.
	Then the following computation shows that $g$ is a left-covariant map.
	\begin{equation*}
	\begin{aligned}
	& \Delta g(\rho \tensora \nu) = \Delta((\id \tensorc \epsilon g)(\Delta_{(\E \tensora \E)}(\rho \tensora \nu)))\\
	=& \Delta(\id \tensorc \epsilon g)(\rho_{(-1)} \nu_{(-1)} \tensorc \rho_{(0)} \tensora \nu_{(0)})\\
	=& \Delta(\rho_{(-1)} \nu_{(-1)}) \epsilon g(\rho_{(0)} \tensora \nu_{(0)})\\
	=& (\rho_{(-1)})_{(1)} (\nu_{(-1)})_{(1)} \tensorc (\rho_{(-1)})_{(2)} (\nu_{(-1)})_{(2)} \epsilon g(\rho_{(0)} \tensora \nu_{(0)})\\
	=& (\rho_{(-1)})_{(1)} (\nu_{(-1)})_{(1)} \tensorc ((\id \tensorc \epsilon g)((\rho_{(-1)})_{(2)} (\nu_{(-1)})_{(2)} \tensorc \rho_{(0)} \tensora \nu_{(0)}))\\
	=& \rho_{(-1)} \nu_{(-1)} \tensorc ((\id \tensorc \epsilon g)(\Delta_{(\E \tensora \E)}(\rho_{(0)} \tensora \nu_{(0)})))\\
	& {\rm \ (where \ we \ have \ used \ {co associativity} \ of \ comodule \ coactions) } \\
	=& \rho_{(-1)} \nu_{(-1)} \tensorc g(\rho_{(0)} \tensora \nu_{(0)})\\
	=& (\id \tensorc g)(\Delta_{(\E \tensora \E)}(\rho \tensora \nu)).
	\end{aligned}
	\end{equation*}
	On the other hand, suppose $g: \E\tensora \E \to \A$ is a left-covariant map. Then we have
	\begin{equation*}
	\begin{aligned}
	&(\id \tensorc \epsilon g) \Delta_{(\E \tensora \E)}(\rho \tensora \nu) = (\id \tensorc \epsilon)(\id \tensorc g)\Delta_{(\E \tensora \E)} (\rho \tensora \nu)\\
	= &(\id \tensorc \epsilon)\Delta g(\rho \tensora \nu) = g(\rho \tensora \nu).
	\end{aligned}
	\end{equation*}
	The proof of the right-covariant case is similar.
\end{proof}	

The following key result will be used throughout the article. 

\begin{lemma} (\cite{heckenberger}) \label{14thfeb20191}
	If $g$ is a pseudo-Riemannian metric which is left-invariant on a left-covariant $\A$-bimodule $\E$, then $g(\omega_1 \tensora \omega_2) \in \IC.1$ for all $\omega_1, \omega_2$ in $\zeroE.$ Similarly, if $ g $ is a right-invariant pseudo-Riemannian metric on a right-covariant $\A$-bimodule, then $ g (\eta_1 \tensora \eta_2) \in \mathbb{C}. 1 $ for all $ \eta_1, \eta_2 $ in $ \E_0. $
\end{lemma}

Let us clarify some of the properties of a left-invariant and right-invariant pseudo-Riemannian metrics.
To that end, we make the next definition which makes sense as we always work with finite bicovariant bimodules (see Remark \ref{20thjune}). The notations used in the next definition will be used throughout the article.

\begin{defn} \label{24thaugust20195}
	Let $ \E $ and $ g $ be as above. For a fixed basis $ \{ \omega_1, \cdots , \omega_n \} $ of $ \zeroE, $ we define $ g_{ij} = g (\omega_i \tensora \omega_j). $ Moreover, we define $ V_g: \E \rightarrow \E^* = \Hom_{\A} ( \E, \A ) $ to be the map defined by the formula
	$$ V_g (e) (f) = g (e \tensora f). $$
\end{defn}

\begin{prop} \label{23rdmay20192}
	Let $g$ be a left-invariant pseudo-Riemannian metric for $\E$ as in Definition \ref{24thmay20191}. Then the following statements hold:
	\begin{itemize}
		\item[(i)] The map $ V_g $ is a one-one right $ \A $-linear map from $ \E $ to $ \E^*. $ 
		\item[(ii)] If $ e \in \E $ is such that $ g (e \tensora f) = 0 $ for all $ f \in {}_0 \E, $ then $ e = 0. $ In particular, the map $ V_{g}|_{{}_0 \E} $ is one-one and hence an isomorphism from $ {}_0 \E $ to $ ({}_0 \E)^*.$
		\item[(iii)] The matrix $((g_{ij}))_{ij}$ is invertible.
		\item[(iv)] Let $g^{ij}$ denote the $(i,j)$-th entry of the inverse of the matrix $((g_{ij}))_{ij}$. Then $g^{ij}$ is an element of $\IC.1$ for all $i,j$.
		\item[(v)] If $g(e \tensora f)=0$ for all $e$ in $\zeroE$, then $f = 0$. 
	\end{itemize}
\end{prop}
\begin{proof}
	The right $ \A $-linearity of $ V_g $ follows from the fact that $g$ is a well-defined map from $\E \tensora \E$ to $\A.$ The condition (2) of Definition \ref{24thmay20191} forces $V_g$ to be one-one. This proves (i). 
	
	For proving (ii), note that $V_g|_{{}_0 \E}$ is the restriction of a one-one map to a subspace. Hence, it is a one-one $\IC$-linear map. Since $g$ is left-invariant, by Lemma \ref{14thfeb20191}, for any $e$ in ${}_0 \E$, $V_g(e)(\zeroE)$ is contained in $\IC$. Therefore, $V_g$ maps $\zeroE$ into $(\zeroE)^*$. Since, $\zeroE$ and $(\zeroE)^*$ have the same finite dimension as vector spaces, $ V_{g}|_{{}_0 \E} : {}_0 \E \to ({}_0 \E)^* $ is an isomorphism. This proves (ii). 
	
	Now we prove (iii). Let $ \{ \omega_i \}_i $ be a basis of $ {}_0 \E $ and $ \{ \omega^*_i \}_i $ be a dual basis, i.e, $ \omega^*_i (\omega_j) = \delta_{ij}. $ Since $ V_{g}|_{{}_0 \E} $ is a vector space isomorphism from $ {}_0 \E $ to $ ({}_0 \E)^* $ by part (ii), there exist complex numbers $ a_{ij} $ such that 
	$$(V_{g})^{-1}(\omega_i^*)= \sum_j a_{ij} \omega_j$$.
	This yields
	\begin{equation*}
	\delta_{ik} = \omega_i^*(\omega_k)
	= g(\sum_j a_{ij} \omega_j \tensora \omega_k)
	= \sum_j a_{ij} g_{jk}.
	\end{equation*}
	Therefore, $((a_{ij}))_{ij}$ is the left-inverse and hence the inverse of the matrix $((g_{ij}))_{ij}$. This proves (iii).
	
	For proving (iv), we use the fact that $g_{ij}$ is an element of $\IC.1$ for all $i,j$. Since
	$$ \sum_k g(\omega_i \tensora \omega_k)g^{kj} = \delta_{ij}.1 = \sum_k g^{ik}g(\omega_k \tensora \omega_j)=\delta_{ij}, $$
	we have 
	$$\sum_k g(\omega_i \tensora \omega_k)\epsilon(g^{kj})= \delta_{ij} = \sum_k \epsilon(g^{ik})g(\omega_k \tensora \omega_j).$$
	So, the matrix $((\epsilon(g^{ij})))_{ij}$ is also an inverse to the matrix $((g(\omega_i \tensora \omega_j)))_{ij}$ and hence $ g^{ij} = \epsilon(g^{ij})$ and $g^{ij}$ is in $\mathbb{C}.1.$
	
	Finally, we prove (v) using (iv). Suppose $f$ be an element in $\E$ such that $g(e \tensora f)=0$ for all $e$ in ${}_0\E$. Let $f=\sum_k \omega_k a_k$ for some elements $a_k$ in $\A$. Then for any fixed index $i_0$, we obtain
	\begin{equation*}
	0 = g(\sum_j g^{i_0 j}\omega_j \tensora \sum_k \omega_k a_k)
	= \sum_k \sum_j g^{i_0j} g_{jk} a_k
	= \sum_k \delta_{i_0k} a_k
	= a_{i_0}.
	\end{equation*}
	Hence, we have that $f=0$. This finishes the proof.
\end{proof}

We apply the results in Proposition \ref{23rdmay20192} to exhibit a basis of the free right $\A$-module $ V_g (\E). $ This will be used in making Definition \ref{26thaugust20191sm} which is needed to prove our main Theorem \ref{29thaugust20191sm}.

\blmma \label{28thaugust2019night1}
Suppose $ \{ \omega_i \}_i $ is a basis of $ \zeroE $ and $ \{ \omega^*_i \}_i $ be the dual basis as in the proof of Proposition \ref{23rdmay20192}. If $ g $ is a pseudo-Riemannian left-invariant metric on $\E,$ then $ V_g (\E) $ is a free right $ \A $-module with basis $ \{ \omega^*_i \}_i. $ 
\elmma 
\begin{proof}
	We will use the notations $ (g_{ij})_{ij} $ and $ g^{ij} $ from of Proposition \ref{23rdmay20192}. Since $ V_g $ is a right $ \A $-linear map, $ V_g (\E) $ is a right $\A$-module. Since
	\begin{equation} \label{28thaugust2019night2} V_g (\omega_i) = \sum_{j} g_{ij} \omega^*_j \end{equation}
	and the inverse matrix $ (g^{ij})_{ij} $ has scalar entries (Proposition \ref{23rdmay20192}), we get
	$$ \omega^*_k = \sum_i g^{ki} V_g (\omega_i) $$
	and so $ \omega^*_k $ belongs to $ V_g (\E) $ for all $k.$ By the right $\A$-linearity of the map $ V_g, $ we conclude that the set $ \{ \omega^*_i \}_i $ is right $\A$-total in $ V_g (\E). $ 
	
	Finally, if $ a_i $ are elements in $ \A $ such that $\sum_k \omega^*_k a_k = 0, $ then by \eqref{28thaugust2019night2}, we have
	$$ 0 = \sum_{i,k} g^{ki} V_g (\omega_i) a_k = V_g (\sum_i \omega_i ( \sum_k g^{ki} a_k ) ). $$
	As $ V_g $ is one-one and $ \{ \omega_i \}_i $ is a basis of the free module $ \E, $ we get
	$$ \sum_k g^{ki} a_k = 0 ~ \forall ~ i. $$
	Multiplying by $ g_{ij} $ and summing over $i$ yields $a_j = 0. $ This proves that $ \{ \omega^*_i \}_i $ is a basis of $ V_g (\E) $ and finishes the proof.
\end{proof}

\brmrk
Let us note that for all $e \in \E,$ the following equation holds:
\begin{equation} \label{25thjune20} e = \sum_i \omega_i \omega^*_i ( e  ). \end{equation}
\ermrk

The following proposition was kindly pointed out to us by the referee for which we will need the notion of a left dual of an object in a monoidal category. We refer to Definition 2.10.1 of \cite{etingof} or Definition XIV.2.1 of \cite{kassel} for the definition.
\begin{prop} \label{25thjune202}
	Suppose $g$ is a pseudo-Riemannian $\A$-bilinear pseudo-Riemannian metric on a finite bicovariant $\A$-bimodule. Let $\widetilde{\E}$ denote the left dual of the object $\E$ in the category $\bimodbicov.$ Then $\widetilde{\E}$ is isomorphic to $\E$ as objects in the category $\bimodbicov$ via the morphism $V_g.$
\end{prop}
\begin{proof} It is well-known that $\widetilde{\E}$ and $ \E^* $ are isomorphic objects in the category $\bimodbicov.$ This follows by using the bicovariant $\A$-bilinear maps
\[ \ev: \tilde{\E} \tensora \E \to \A; \quad \phi \tensora e \mapsto \phi(e), ~ \coev: \A \to \E \tensora \tilde{\E}; \quad 1 \mapsto \sum_i \omega_i \tensora \omega_i^* \]
We define $ \ev_g: \E \tensora \E \rightarrow \A  $ and  $ \coev_g: \A \rightarrow \E \tensora \E $ by the following formulas:$$ \ev_g ( e \tensora f ) = g ( e \tensora f ), \quad \coev_g ( 1 ) = \sum_i \omega_i \tensora V^{-1}_g ( \omega^*_i ). $$
Then since $g$ is both left and right $\A$-linear, $ \ev_g $ and $ \coev_g $ are $\A$-$\A$-bilinear maps. The bicovariance of $g$ implies the bicovariance of $\ev_g$ while the bicovariance of $\coev_g = ( \id \tensora V^{-1}_g ) \circ \coev $ follows from the bicovariance of $V_g$ and $\coev.$

Since the left dual of an object is unique upto isomorphism, we need to check the following identities for all $e$ in $\E$:
$$ ( \ev_g \tensora \id ) ( \id \tensora \coev_g ) ( e ) = e, ~ ( \id \tensora \ev_g ) ( \coev_g \tensora \id ) ( e ) = e.   $$
But these follow by a simple computation using the fact that ${}_0 \E $ is right $\A$-total in $\E$ and the identity \eqref{25thjune20}.

From the above discussion, we have that $\E$ and $\E^*$ are two left duals of the object $\E$ in the category $ \bimodbicov. $ Then by the proof of Proposition 2.10.5 of \cite{etingof}, we know that $ ( \ev_g \tensora \id_{\E^*} ) ( \id_{\E} \tensora \coev ) $ is an isomorphism from $\E$ to $\E^*.$ But it can be easily checked that $ ( \ev_g \tensora \id_{\E^*} ) ( \id_{\E} \tensora \coev ) = V_g. $ This completes the proof.
\end{proof}

Now we state a result on bi-invariant (i.e both left and right-invariant) pseudo-Riemannian metric.

\begin{prop} \label{27thjune20197}
	Let $g$ be a pseudo-Riemannian metric on $\E$ and the symbols $\{ \omega_i \}_i$, $\{ g_{ij} \}_{ij}$ be as above. If
	\begin{equation} \label{29thaugust20194}
	{}_\E \Delta(\omega_i) = \sum_j \omega_j \tensorc R_{ji}
	\end{equation}
	(see \eqref{26thaugust20191}), then $g$ is bi-invariant if and only if the elements $ g_{ij} $ are scalar and
	\begin{equation} \label{6thnov20191}
	g_{ij} = \sum_{kl} g_{kl} R_{ki} R_{lj}.
	\end{equation}
\end{prop}
\begin{proof}
	Since $g$ is left-invariant, the elements $g_{ij}$ are in $\IC.1$. Moreover, the right-invariance of $g$ implies that $g$ is right-covariant (Proposition \ref{29thaugust20192}), i.e. \begin{equation*} \begin{aligned} & 1 \tensorc g_{ij} = \Delta(g_{ij}) = (g \tensora \id){}_{\E \tensora \E} \Delta(\omega_i \tensorc \omega_j) \\ =& (g \tensora \id)(\sum_{kl} \omega_k \tensora \omega_l \tensorc R_{ki} R_{lj}) = 1 \tensorc \sum_{kl} g_{kl} R_{ki} R_{lj}, \end{aligned} \end{equation*} so that
	\begin{equation} \label{29thaugust20193}
	g_{ij} = \sum_{kl} g_{kl} R_{ki} R_{lj}.
	\end{equation}
	Conversely, if $ g_{ij} = g (\omega_i \tensora \omega_j) $ are scalars and \eqref{6thnov20191} is satisfied, then $ g $ is left-invariant and right-covariant. By Proposition \ref{29thaugust20192}, $g$ is right-invariant. 
\end{proof}	

We end this section by comparing our definition of pseudo-Riemannian metrics with some of the other definitions available in the literature.

Proposition \ref{23rdmay20192} shows that our notion of pseudo-Riemannian metric coincides with the right $\A$-linear version of a ``symmetric metric" introduced in Definition 2.1 of \cite{heckenberger} if we impose the condition of left-invariance.

Next, we compare our definition with the one used by Beggs and Majid in Proposition 4.2 of \cite{majidpodles} (also see \cite{beggsmajidbook} and references therein). To that end, we need to recall the construction of the two forms by Woronowicz (\cite{woronowicz}).

If $ \E $ is a bicovariant $\A$-bimodule and $ \sigma $ be the map as in Proposition \ref{4thmay20193}, Woronowicz defined the space of two forms as: 
$$\twoform := (\E \tensora \E) \big/ \rm {\rm Ker} (\sigma - 1).$$
The symbol $\wedge$ will denote the quotient map
$$ \wedge: \E \tensora \E \to \twoform. $$
Thus,
\begin{equation} \label{22ndaugust20191}
{\rm Ker}(\wedge) = {\rm Ker}(\sigma - 1). 
\end{equation}
In Proposition 4.2 of \cite{majidpodles}, the authors define a metric on a bimodule $ \E $ over a (possibly) noncommutative algebra $\A$ as an element $ h $ of $ \E \tensora \E $ such that $ \wedge(h) = 0. $ We claim that metrics in the sense of Beggs and Majid are in one to one correspondence with elements $ g \in \Hom_\A (\E \tensora \E, \A) $ (not necessarily left-invariant) such that $ g \circ \sigma = g. $ Thus, modulo the nondegeneracy condition (ii) of Definition \ref{24thmay20191}, our notion of pseudo-Riemannian metric matches with the definition of metric by Beggs and Majid.

Indeed, if $ g \in \Hom_\A (\E \tensora \E, \A) $ as above and $ \{ \omega_i \}_i, $ is a basis of $\zeroE,$ then the equation $ g \circ \sigma = g $ implies that 
$$ g \circ \sigma (\omega_i \tensora \omega_j) = g (\omega_i \tensora \omega_j). $$
However, by equation (3.15) of \cite{woronowicz}, we know that 
$$ \sigma (\omega_i \tensora \omega_j) = \sum_{k,l} \sigma^{kl}_{ij} \omega_k \tensora \omega_l $$
for some scalars $ \sigma^{kl}_{ij}. $ Therefore, we have
\begin{equation} \label{12thoct20191} \sum_{k,l} \sigma^{kl}_{ij} g (\omega_k \tensora \omega_l) = g (\omega_i \tensora \omega_j). \end{equation}
We claim that the element $ h = \sum_{i,j} g (\omega_i \tensora \omega_j) \omega_i \tensora \omega_j $ satisfies $ \wedge (h) = 0. $ Indeed, by virtue of \eqref{22ndaugust20191}, it is enough to prove that $ (\sigma - 1) (h) = 0. $ But this directly follows from \eqref{12thoct20191} using the left $\A$-linearity of $\sigma.$ 

This argument is reversible and hence starting from $ h \in \E \tensora \E $ satisfying $ \wedge(h) = 0, $ we get an element $ g \in \Hom_\A (\E \tensora \E, \A) $ such that for all $i,j,$
$$ g \circ \sigma (\omega_i \tensora \omega_j) = g (\omega_i \tensora \omega_j). $$
Since $ \{ \omega_i \tensora \omega_j : i,j \} $ is right $\A$-total in $ \E \tensora \E $ (Corollary \ref{3rdaugust20191}) and the maps $g, \sigma $ are right $\A$-linear, we get that $ g \circ \sigma = g.$ This proves our claim. Let us note that since we did not assume $g$ to be left invariant, the quantities $ g (\omega_i \tensora \omega_j)$ need not be scalars. However, the proof goes through since the elements $ \sigma^{ij}_{kl} $ are scalars.

\section{Pseudo-Riemannian metrics for cocycle deformations} \label{21staugust20197}

This section concerns the braiding map and pseudo-Riemannian metrics of bicovariant bimodules under cocycle deformations of Hopf algebras. This section contains two main results. We start by recalling that a bicovariant bimodule $ \E $ over a Hopf algebra $\A$ can be deformed in the presence of a $2$-cocycle $ \gamma $ on $\A$ to a bicovariant $\A_\gamma$-bimodule $\E_\gamma.$ We prove that the canonical braiding map of the bicovariant bimodule $ \E_\gamma $ (Proposition \ref{4thmay20193}) is a cocycle deformation of the canonical braiding map of $\E.$ Finally, we prove that pseudo-Riemannian bi-invariant metrics on $ \E $ and $ \E_\gamma $ are in one to one correspondence.

Throughout this section, we will make heavy use of the Sweedler notations as spelled out in \eqref{28thaugust20191}, \eqref{28thaugust20192} and \eqref{28thaugust20193}. The coassociativity of $\Delta$ will be expressed by the following equation:
$$(\Delta \tensorc \id)\Delta(a) = (\id \tensorc \Delta) \Delta(a) = a_{(1)} \tensorc a_{(2)} \tensorc a_{(3)}.$$ 
Also, when $m$ is an element of a bicovariant bimodule, we will use the notation
\begin{equation} \label{28thaugust20194}
(\id \tensorc {}_M \Delta)\Delta_M (m) = (\Delta_M \tensorc \id) {}_M \Delta(m) = m_{(-1)} \tensorc m_{(0)} \tensorc m_{(1)}.
\end{equation}

\bdfn
A cocycle $\gamma$ on a Hopf algebra $ ( \A, \Delta ) $ is a $\mathbb{C}$-linear map $\gamma: \A \tensorc \A \rightarrow \mathbb{C} $ such that it is convolution invertible, unital, i.e, 
$$ \gamma (a \tensorc 1) = \epsilon (a) = \gamma (1 \tensorc a) $$
and for all $a,b,c$ in $\A,$
\begin{equation} \label{(iii)} \gamma (a_{(1)} \tensorc b_{(1)}) \gamma (a_{(2)} b_{(2)} \tensorc c) = \gamma (b_{(1)} \tensorc c_{(1)}) \gamma (a \tensorc b_{(2)} c_{(2)}).\end{equation}
\edfn

Given a Hopf algebra $ (\A, \Delta) $ and such a cocycle $ \gamma $ as above, we have a new Hopf algebra $ (\A_\gamma, \Delta_\gamma) $ which is equal to $ \A $ as a vector space, the coproduct $ \Delta_\gamma $ is equal to $ \Delta $ while the algebra structure $ \ast_\gamma $ on $ \A_\gamma $ is defined by the following equation:
\begin{equation} \label{2ndseptember20191} a \ast_\gamma b = \gamma (a_{(1)} \tensorc b_{(1)}) a_{(2)} b_{(2)} \overline{\gamma} (a_{(3)} \tensorc b_{(3)}). \end{equation}
Here, $ \overline{\gamma} $ is the convolution inverse to $ \gamma $ which is unital and satisfies the following equation: 
\begin{equation} \label{(iiiprime)} \overline{\gamma} (a_{(1)} b_{(1)} \tensorc c) \overline{\gamma} (a_{(2)} \tensorc b_{(2)}) = \overline{\gamma}(a \tensorc b_{(1)} c_{(1)}) \overline{\gamma} (b_{(2)} \tensorc c_{(2)}). \end{equation}

We refer to Theorem 1.6 of \cite{doi} for more details.

Suppose $ M $ is a bicovariant $\A$-$\A$-bimodule. Then $ M $ can also be deformed in the presence of a cocycle. This is the content of the next proposition.

\begin{prop} \label{8thmay20191} (Theorem 2.5 of \cite{majidcocycle}) 
	Suppose $ M $ is a bicovariant $\A$-bimodule and $ \gamma $ is a $2$-cocycle on $\A.$ Then we have a bicovariant $\A_\gamma $-bimodule $ M_\gamma $ which is equal to $ M $ as a vector space but the left and right $ \A_\gamma $-module structures are defined by the following formulas:
	\begin{eqnarray} & \label{9thmay20197} a *_\gamma m = \gamma(a_{(1)} \tensorc m_{(-1)}) a_{(2)}  m_{(0)} \overline{\gamma}(a_{(3)} \tensorc m_{(1)})\\ & \label{9thmay20198} m *_\gamma a = \gamma(m_{(-1)} \tensorc a_{(1)}) m_{(0)}  a_{(2)} \overline{\gamma}(m_{(1)} \tensorc a_{(3)}), \end{eqnarray}
	for all elements $m$ of $M$ and for all elements $a$ of $\A$. Here, $*_\gamma$ denotes the right and left $\A_\gamma$-module actions, and $.$ denotes the right and left $\A$-module actions.
	
	The $\A_\gamma$-bicovariant structures are given by \begin{equation} \label{10thseptember20191sm} {} \Delta_{M_\gamma}:= \Delta_M : M_\gamma \rightarrow \A_\gamma \tensorc M_\gamma ~ {\rm and} ~ {}_{M_\gamma} \Delta:= {}_M \Delta: M_\gamma \rightarrow M_\gamma \tensorc \A_\gamma. \end{equation} 
\end{prop}

\begin{rmk} \label{20thjune2}
	From Proposition \ref{8thmay20191}, it is clear that if $M$ is a finite bicovariant bimodule (see Remark \ref{20thjune}), then $ M_\gamma $ is also a finite bicovariant bimodule.
\end{rmk}

We end this subsection by recalling the following result on the deformation of bicovariant maps.

\begin{prop} \label{11thjuly20192} (Theorem 2.5 of \cite{majidcocycle}) 
	Let $(M, \Delta_M, {}_M \Delta)$ and $(N, \Delta_N, {}_N \Delta)$ be bicovariant $\A$-bimodules, $T: M \to N$ be a $\IC$-linear bicovariant map and $\gamma$ be a cocycle as above. Then there exists a map $T_\gamma: M_\gamma \to N_\gamma$ defined by $T_\gamma (m) = T(m)$ for all $m$ in $M$. Thus, $T_\gamma = T$ as $\IC$-linear maps. Moreover, we have the following:
	\begin{itemize}
		\item[(i)] the deformed map $T_\gamma: M_\gamma \to N_\gamma$ is an $\A_\gamma$ bicovariant map,
		\item[(ii)] if $T$ is a bicovariant right (respectively left) $\A$-linear map, then $T_\gamma$ is a bicovariant right (respectively left) $\A_\gamma$-linear map,
		\item[(iii)] if $(P, \Delta_P, {}_{P} \Delta)$ is another bicovariant $\A$-bimodule, and $S: N \to {P} $ is a bicovariant map, then $(S \circ T)_\gamma : M_\gamma \to P_\gamma$ is a bicovariant map and $S_\gamma \circ T_\gamma = (S \circ T)_\gamma$.
	\end{itemize} 
\end{prop}

\subsection{Deformation of the braiding map} \label{18thseptember20192}

Suppose $ \E $ is a bicovariant $\A$-bimodule, $ \sigma $ be the bicovariant braiding map of Proposition \ref{4thmay20193} and $ g $ be a bi-invariant metric. Then Proposition \ref{11thjuly20192} implies that we have deformed maps $ \sigma_\gamma$ and $ g_\gamma. $ In this subsection, we study the map $ \sigma_\gamma. $ The map $ g_\gamma $ will be discussed in the next subsection. We will need the following result:
\begin{prop} \label{11thjuly20191} (Theorem 2.5 of \cite{majidcocycle})
	Let $(M, \Delta_{M}, {}_M\Delta)$ and $(N, \Delta_{N}, {}_N\Delta)$ be bicovariant bimodules over a Hopf algebra $\A$ and $\gamma$ be a cocycle as above. Then there exists a bicovariant $\A_\gamma$- bimodule isomorphism 
	$$ \xi: M_\gamma \otimes_{\A_\gamma} N_\gamma \rightarrow (M \tensora N)_\gamma. $$
	The isomorphism $\xi$ and its inverse are respectively given by
	\begin{equation*}
	\begin{aligned}
	\xi (m \otimes_{\A_\gamma} n) & = \gamma(m_{(-1)} \tensorc n_{(-1)}) m_{(0)} \tensora n_{(0)} \overline{\gamma}(m_{(1)} \tensorc n_{(1)})\\
	\xi^{-1} (m \tensora n) & = \overline{\gamma}(m_{(-1)} \tensorc n_{(-1)}) m_{(0)} \otimes_{\A_\gamma} n_{(0)} {\gamma}(m_{(1)} \tensorc n_{(1)})
	\end{aligned}
	\end{equation*}
\end{prop}
As an illustration, we make the following computation which will be needed later in this subsection:
\begin{lemma} \label{14thseptember20191}
	Suppose $ \omega \in \zeroE, \eta \in \Ezero. $ Then the following equation holds:
	\begin{eqnarray*}
		\xi^{-1} (\gamma (\eta_{(-1)} \tensorc 1) \eta_{(0)} \tensora \omega_{(0)} \overline{\gamma} (1 \tensorc \omega_{(1)})) &=& \eta \otimes_{\A_\gamma} \omega.
	\end{eqnarray*}
\end{lemma}
\begin{proof}
	Let us first clarify that we view $ \gamma (\eta_{(-1)} \tensorc 1) \eta_{(0)} \tensora \omega_{(0)} \overline{\gamma} (1 \tensorc \omega_{(1)}) $
	as an element in $ (\E \tensora \E)_\gamma. $ 
	Then the equation holds because of the following computation:
	\begin{equation*}
	\begin{aligned}
	&\xi^{-1} (\gamma (\eta_{(-1)} \tensorc 1) \eta_{(0)} \tensora \omega_{(0)} \overline{\gamma} (1 \tensorc \omega_{(1)}))\\ =& \gamma (\eta_{(-1)} \tensorc 1) \xi^{-1} ( \eta_{(0)} \tensora \omega_{(0)}) \overline{\gamma} (1 \tensorc \omega_{(1)}) \\
	=& \gamma (\eta_{(-2)} \tensorc 1) \overline{\gamma}(\eta_{(-1)} \tensorc 1) \eta_{(0)} \otimes_{\A_\gamma} \omega_{(0)} \gamma(1 \tensorc \omega_{(1)}) \overline{\gamma} (1 \tensorc \omega_{(2)})\\ ~ &{\rm (} {\rm since} ~ \omega \in \zeroE, ~ \eta \in \Ezero {\rm)}\\
	=& \epsilon(\eta_{(-2)}) \epsilon(\eta_{(-1)}) \eta_{(0)} \otimes_{\A_\gamma} \omega_{(0)} \epsilon (\omega_{(1)}) \epsilon (\omega_{(2)}) \ {\rm (since \ \overline{\gamma} \ and \ \gamma \ are \ normalised)}\\
	=& \eta \otimes_{\A_\gamma} \omega.
	\end{aligned}
	\end{equation*}
\end{proof}
Now, we are in a position to study the map $ \sigma_\gamma. $ By Proposition \ref{8thmay20191}, $ \E_\gamma $ is a bicovariant $ \A_\gamma $-bimodule. Then Proposition \ref{4thmay20193} guarantees the existence of a canonical braiding from $\E_\gamma \otimes_{\A_\gamma} \E_\gamma$ to itself. We show that this map is nothing but the deformation $\sigma_\gamma$ of the map $\sigma$ associated with the bicovariant $\A$-bimodule $\E$. By the definition of $ \sigma_\gamma, $ it is a map from $ (\E \tensora \E)_\gamma $ to $ (\E \tensora \E )_\gamma. $ However, by virtue of Proposition \ref{11thjuly20191}, the map $ \xi $ defines an isomorphism from $ \E_\gamma \otimes_{\A_\gamma} \E_\gamma $ to $ (\E \tensora \E )_\gamma. $ By an abuse of notation, we will denote the map 
$$ \xi^{-1} \sigma_\gamma \xi : \E_\gamma \otimes_{\A_\gamma} \E_\gamma \rightarrow \E_\gamma \otimes_{\A_\gamma} \E_\gamma $$
by the symbol $ \sigma_\gamma $ again. 

\begin{thm} \label{28thaugust20196} (Theorem 2.5 of \cite{majidcocycle})
	Let $\E$ be a bicovariant $\A$-bimodule and $\gamma$ be a cocycle as above. Then the deformation $\sigma_\gamma$ of $\sigma$ is the unique bicovariant $\A_\gamma$-bimodule braiding map on $\E_\gamma$ given by Proposition \ref{4thmay20193}.
\end{thm}
\begin{proof}
	Since $\sigma$ is a bicovariant $\A$-bimodule map from $\E \tensora \E$ to itself, part (ii) of Proposition \ref{11thjuly20192} implies that $\sigma_\gamma$ is a bicovariant $\A_\gamma$-bimodule map from $(\E \tensora \E)_\gamma \cong \E_\gamma \otimes_{\A_\gamma} \E_\gamma$ to itself. By Proposition \ref{4thmay20193}, there exists a unique $\A_\gamma$-bimodule map $\sigma^\prime$ from $\E_\gamma \otimes_{\A_\gamma} \E_\gamma$ to itself such that $\sigma^\prime(\omega \otimes_{\A_\gamma} \eta) = \eta \otimes_{\A_\gamma} \omega$ for all $\omega$ in ${}_0(\E_\gamma)$, $\eta$ in $(\E_\gamma)_0$. 
	
	Since ${}_0(\E_\gamma) = \zeroE$ and $(\E_\gamma)_0 = \Ezero$, it is enough to prove that $\sigma_\gamma(\omega \otimes_{\A_\gamma} \eta) = \eta \otimes_{\A_\gamma} \omega$ for all $\omega$ in $\zeroE$, $\eta$ in $\Ezero$.
	
	We will need the concrete isomorphism between $\E_\gamma \otimes_{\A_\gamma} \E_\gamma$ and $(\E \tensora \E)_\gamma$ defined in Proposition \ref{11thjuly20191}. Since $\omega$ is in $\zeroE$ and $\eta$ is in $\Ezero$, this isomorphism maps the element $\omega \otimes_{\A_\gamma} \eta$ to $\gamma(1 \tensorc \eta_{(-1)}) \omega_{(0)} \tensora \eta_{(0)} \overline{\gamma}(\omega_{(1)} \tensorc 1)$. Then, by the definition of $\sigma_\gamma$, we compute the following:
	\begin{equation*}
	\begin{aligned}
	& \sigma_\gamma(\omega \otimes_{\A_\gamma} \eta)
	= \sigma(\gamma(1 \tensorc \eta_{(-1)}) \omega_{(0)}\tensora \eta_{(0)} \overline{\gamma} (\omega_{(1)} \tensorc 1)) \\
	=& \sigma(\epsilon (\eta_{(-1)}) \omega_{(0)}\tensora \eta_{(0)} \epsilon(\omega_{(1)}))
	= \epsilon (\eta_{(-1)}) \eta_{(0)}\tensora \omega_{(0)} \epsilon(\omega_{(1)})\\
	=& \gamma (\eta_{(-1)} \tensorc 1) \eta_{(0)}\tensora \omega_{(0)} \overline{\gamma}(1 \tensorc \omega_{(1)})
	= \eta \otimes_{\A_\gamma} \omega,
	\end{aligned}	
	\end{equation*}
	where, in the last step we have used Lemma \ref{14thseptember20191}.
\end{proof}

\brmrk \label{26thjune20202}
Proposition \ref{8thmay20191}, Proposition \ref{11thjuly20192}, Proposition \ref{11thjuly20191} and Theorem \ref{28thaugust20196} together imply that the categories $ \bimodbicov $ and $\bimodbicovgamma$ are isomorphic as braided monoidal categories. This was the content of Theorem 2.5 of \cite{majidcocycle}. The referee has pointed out that this is a special case of a much more generalized result of Bichon (Theorem 6.1 of \cite{bichon}) which says that if two Hopf algebras are monoidally equivalent, then the corresponding categories of right-right Yetter Drinfeld modules are also monoidally equivalent.

However, in Theorem \ref{28thaugust20196}, we have proved in addition that the braiding on $\bimodbicovgamma$ is precisely the Woronowicz braiding of Proposition \ref{4thmay20193}.
\ermrk

\begin{corr} \label{18thsep20194}
	If the unique bicovariant $\A$-bimodule braiding map $\sigma$ for a bicovariant $\A$-bimodule $\E$ satisfies the equation $\sigma^2 = 1$, then the bicovariant $\A_\gamma$-bimodule braiding map $ \sigma_\gamma$ for the bicovariant $\A_\gamma$-bimodule $\E_\gamma$ also satisfies $\sigma_\gamma^2 = 1$.
	
	In particular, if $\A$ is the commutative Hopf algebra of regular functions on a compact semisimple Lie group $G$ and $\E$ is its canonical space of one-forms, then the braiding map $\sigma_\gamma$ for $\E_\gamma$ satisfies $\sigma_\gamma^2 = 1$.
\end{corr}
\begin{proof}
	By Theorem \ref{28thaugust20196}, $\sigma_\gamma$ is the unique braiding map for the bicovariant $\A_\gamma$-bimodule $\E_\gamma$. Since, by our hypothesis, $\sigma^2 = 1$, the deformed map $\sigma_\gamma$ also satisfies $\sigma_\gamma^2 = 1$ by part (iii) of Proposition \ref{11thjuly20192}.\\
	Next, if $\A$ is a commutative Hopf algebra as in the statement of the corollary and $\E$ is its canonical space of one-forms, then we know that the braiding map $\sigma$ is just the flip map, i.e. for all $e_1, e_2$ in $\E$,
	\[ \sigma(e_1 \tensora e_2) = e_2 \tensora e_1, \]
	and hence it satisfies $\sigma^2 = 1$. Therefore, for every cocycle deformation $\E_\gamma$ of $\E$, the corresponding braiding map satisfies $\sigma_\gamma^2 = 1$.
\end{proof}

\subsection{Pseudo-Riemannian bi-invariant metrics on $\E_\gamma$}

Suppose $ \E $ is a bicovariant $\A$-bimodule and $\E_\gamma $ be its cocycle deformation as above. The goal of this subsection is to prove that a pseudo-Riemannian bi-invariant metric on $ \E $ naturally deforms to a pseudo-Riemannian bi-invariant metric on $\E_\gamma. $ Since $ g $ is a bicovariant (i.e, both left and right covariant) map from the bicovariant bimodule $ \E \tensora \E $ to itself, then by Proposition \ref{11thjuly20192}, we have a right $ \A_\gamma $-linear bicovariant map $ g_{\gamma} $ from $ \E_\gamma \otimes_{\A_\gamma} \E_\gamma $ to itself. We need to check the conditions (i) and (ii) of Definition \ref{24thmay20191} for the map $ g_\gamma. $

The proof of the equality $ g_\gamma = g_\gamma \circ \sigma_\gamma $ is straightforward. However, checking condition (ii), i.e, verifying that the map $ V_{g_\gamma} $ is an isomorphism onto its image needs some work. The root of the problem is that we do not yet know whether $ \E^* = V_g ( \E ).$ Our strategy to verify condition (ii) is the following: we show that the right $ \A $-module $ V_g ( \E ) $ is a bicovariant right $\A$-module (see Definition \ref{28thaugust20195}) in a natural way. Let us remark that since the map $ g $ (hence $ V_g $) is not left $\A$-linear, $ V_g ( \E ) $ need not be a left $\A$-module. Since bicovariant right $ \A $-modules and bicovariant maps between them can be deformed (Proposition \ref{28thaugust20191sm}), the map $ V_g $ deforms to a right $ \A_\gamma $-linear isomorphism from $ \E_\gamma $ to $ (V_g (\E) )_\gamma. $ Then in Theorem \ref{29thaugust20191sm}, we show that $ (V_g)_\gamma $ coincides with the map $ V_{g_\gamma} $ and the latter is an isomorphism onto its image. This is the only subsection where we use the theory of bicovariant right modules (as opposed to bicovariant bimodules). 

\begin{defn} \label{28thaugust20195}
	Let $M$ be a right $\A$-module, and $\Delta_M : M \to \A \tensorc M$ and ${}_M \Delta: M \to M \tensorc \A$ be $\mathbb{C}$-linear maps. We say that $(M,\Delta_M, {}_M \Delta)$ is a bicovariant right $\A$-module if the triplet is an object of the category $ \rightmodbicov, $ i.e, 
	\begin{itemize}
		\item[(i)] $(M, \Delta_M)$ is a left $\A$-comodule,\\
		\item[(ii)] $(M, {}_M \Delta)$ is a right $\A$-comodule,\\
		\item[(iii)] $ (\id \tensorc {}_M\Delta) \Delta_M = (\Delta_M \tensorc \id) {}_M\Delta $, \\
		\item[(iv)] For any $a$ in $\A$ and $m$ in $M$, $$ \Delta_M(m a)=\Delta_M(m)\Delta(a), \quad {}_M \Delta(ma) = {}_M \Delta(m) \Delta(a).$$
	\end{itemize}
\end{defn}	

For the rest of the subsection, $ \E $ will denote a bicovariant $ \A $-bimodule. Moreover, $ \{ \omega_i \}_i $ will denote a basis of $ \zeroE $ and $ \{ \omega^*_i \}_i $ the dual basis, i.e, $ \omega^*_i (\omega_j) = \delta_{ij}. $

Let us recall that \eqref{26thaugust20191} implies the existence of elements $R_{ij}$ in $\A$ such that
\begin{equation} \label{26thaugust20192}
{}_\E \Delta(\omega_i) = \sum_{ij} \omega_j \tensorc R_{ji}.
\end{equation}
We want to show that $ V_g (\E) $ is a bicovariant right $\A$-module. To this end, we recall that (Lemma \ref{28thaugust2019night1} ) $ V_g (\E) $ is a free right $ \A $-module with basis $ \{ \omega^*_i \}_i. $ This allows us to make the following definition.

\begin{defn} \label{26thaugust20191sm}
	Let $\{ \omega_i \}_i$ and $ \{ \omega^*_i \}_i $ be as above and $g$ a bi-invariant pseudo-Riemannian metric on $\E $. Then we can endow $V_g(\E)$ with a left-coaction $\Delta_{V_g(\E)} : V_g(\E) \to \A \tensorc V_g(\E)$ and a right-coaction ${}_{V_g(\E)} \Delta : V_g(\E) \to V_g(\E) \tensorc \A$, defined by the formulas
	\begin{equation} \label{29thaugust20192jb}
	\Delta_{V_g(\E)}(\sum_i \omega^\ast_i a_i) = \sum_i (1 \tensorc \omega^\ast_i) \Delta(a_i), ~	
	{}_{V_g(\E)} \Delta(\sum_i \omega^\ast_i a_i) = \sum_{ij} (\omega^\ast_j \tensorc S(R_{ij}))\Delta(a_i),
	\end{equation}
	where the elements $R_{ij}$ are as in \eqref{26thaugust20192}.
\end{defn}
Then we have the following result.

\begin{prop} \label{21staugust20191}
	The triplet $ (V_g (\E), \Delta_{V_g(\E)}, {}_{V_g(\E)} \Delta ) $ is a bicovariant right $ \A $-module. Moreover, the map $ V_g: \E \rightarrow V_g (\E) $ is bicovariant, i.e, we have
	\begin{equation} \label{29thaugust20191}
	\Delta_{V_g(\E)} (V_g(e)) = (\id \tensorc V_g) \Delta_\E(e), ~
	{}_{V_g(\E)} \Delta (V_g(e)) = (V_g \tensorc \id) {}_\E \Delta(e).
	\end{equation}
\end{prop}
\begin{proof} The fact that $ (V_g (\E), \Delta_{V_g(\E)}, {}_{V_g(\E)} \Delta ) $ is a bicovariant right $ \A $-module follows immediately from the definition of the maps $ \Delta_{V_g(\E)} $ and $ {}_{V_g(\E)} \Delta.$ So we are left with proving \eqref{29thaugust20191}. 
	Let $ e \in \E. $ Then there exist elements $ a_i $ in $ \A $ such that $ e = \sum_i \omega_i a_i. $ Hence, by \eqref{28thaugust2019night2}, we obtain
	\begin{eqnarray*}
		\Delta_{V_g (\E)} (V_g (e) ) &=&	\Delta_{V_g(\E)} (V_g(\sum_i \omega_i a_i)) = \Delta_{V_g(\E)} (\sum_{ij} g_{ij} \omega^\ast_j a_i) \\ &=& \sum_{ij} (1 \tensorc g_{ij} \omega^\ast_j) \Delta(a_{i})
		=\sum_{i} ((\id \tensorc V_g)(1 \tensorc \omega_i)) \Delta(a_i)\\ &=& \sum_{i} (\id \tensorc V_g)(\Delta_{\E}(\omega_i)) \Delta(a_i)\\
		&=& \sum_{i} (\id \tensorc V_g)\Delta_{\E}(\omega_i a_i) = (\id \tensorc V_g ) \Delta_\E (e).
	\end{eqnarray*}
	This proves the first equation of \eqref{29thaugust20191}.
	
	For the second equation, we begin by making an observation. Since $ {}_\E \Delta (\omega_i) = \sum_j \omega_j \tensorc R_{ji}, $ we have
	\begin{equation*}
	\delta_{ij} = \epsilon (R_{ij}) = m(S \tensorc \id) \Delta(R_{ij}) = \sum_k S(R_{ik}) R_{kj}.
	\end{equation*}
	Therefore, multiplying \eqref{29thaugust20193} by $S(R_{jm})$ and summing over $j$, we obtain
	\begin{equation} \label{6thnov20192} \sum_j g_{ij} S(R_{jm}) = \sum_j g_{jm} R_{ji}.\end{equation} 
	Now by using \eqref{28thaugust2019night2}, we compute
	\begin{eqnarray*}
		{}_{V_g(\E)} \Delta (V_g (e)) &=&	{}_{V_g(\E)} \Delta (V_g(\sum_i \omega_i a_i)) = {}_{V_g(\E)} \Delta (\sum_{ij} g_{ij} \omega^\ast_j a_i)\\ &=& \sum_{ij} {}_{V_g(\E)} \Delta (g_{ij} \omega^\ast_j) \Delta({a_i})
		= \sum_{ijk} g_{ij} \omega^\ast_k \tensorc S(R_{jk}) \Delta({a_i})\\ &=& \sum_{ik} \omega^\ast_k \tensorc \sum_j g_{ij} S(R_{jk}) \Delta({a_i })\\
		&=& \sum_{ik} \omega^\ast_k \tensorc \sum_j g_{jk} R_{ji} \Delta (a_i) ~ {\rm (} {\rm by} ~ \eqref{6thnov20192} {\rm)}\\
		&=& \sum_{ijk} g_{jk} \omega^\ast_k \tensorc R_{ji} \Delta (a_i)= \sum_{ij} V_g(\omega_j) \tensorc R_{ji} \Delta (a_i)\\
		&=& \sum_i (V_g \tensorc \id){}_{\E} \Delta (\omega_i) \Delta (a_i) ~ {\rm (} {\rm by} ~ \eqref{26thaugust20192} {\rm)}\\
		&=& \sum_i (V_g \tensorc \id){}_{\E} \Delta (\omega_i a_i) 	= (V_g \tensorc \id){}_{\E} \Delta (e).
	\end{eqnarray*}
	This finishes the proof.
\end{proof}

Now we recall that bicovariant right $\A$-modules (i.e, objects of the category $\rightmodbicov$) can be deformed too.

\begin{prop} \label{28thaugust20191sm} (Theorem 5.7 of \cite{schauenberg})
	Let $(M, \Delta_M, {}_M \Delta)$ be a bicovariant right $\A$-module and $\gamma$ be a 2-cocycle on $\A$. Then
	\begin{itemize} 
		\item[(i)] $M$ deforms to a bicovariant right $\A_\gamma$-module, denoted by $M_\gamma$,\\
		\item[(ii)] if $(N, \Delta_N, {}_N \Delta)$ is another bicovariant right $\A$-module and $T: M \to N$ is a bicovariant right $\A$-linear map, then the deformation $T_\gamma : M_\gamma \to N_\gamma$ is a bicovariant right $\A_\gamma$-linear map,		\item[(iii)] $T_\gamma$, as in (ii), is an isomorphism if and only if $T$ is an isomorphism.
	\end{itemize}
\end{prop}
\begin{proof}
	Parts (i) and (ii) follow from the equivalence of categories $ \leftcov $ and $\leftcovgamma$ combined with the $\rightmodbicov$ analogue of (non-monoidal part of) the second assertion of Proposition 5.7 of \cite{schauenberg}. Part (iii) follows by noting that since the map $T$ is a bicovariant right $\A$-linear map, its inverse $T^{-1}$ is also a bicovariant right $\A$-linear map. Thus, the deformation $(T^{-1})_\gamma$ of $T^{-1}$ exists and is the inverse of the map $T_\gamma$.
\end{proof}	
As an immediate corollary, we make the following observation. 
\begin{corr} \label{28thaugust20192sm}
	Let $g$ be a bi-invariant pseudo-Riemannian metric on a bicovariant $\A$-bimodule $\E$. Then the following map is a well-defined isomorphism.
	$$(V_g)_\gamma: \E_\gamma \rightarrow (V_g(\E))_\gamma = (V_g)_{\gamma} (\E_\gamma) $$ 
\end{corr}
\begin{proof}
	Since both $\E$ and $V_g(\E)$ are bicovariant right $\A$-modules, and $V_g$ is a right $\A$-linear bicovariant map (Proposition \ref{21staugust20191}), Proposition \ref{28thaugust20191sm} guarantees the existence of $(V_g)_\gamma$. Since $g$ is a pseudo-Riemannian metric, by (ii) of Definition \ref{24thmay20191}, $V_g: \E \to V_g(\E)$ is an isomorphism. Then, by (iii) of Proposition \ref{28thaugust20191sm}, $(V_g)_\gamma$ is also an isomorphism.
\end{proof}
Now we are in a position to state and prove the main result of this section which shows that there is an abundant supply of bi-invariant pseudo-Riemannian metrics on $ \E_\gamma.$ Since $ g $ is a map from $ \E \tensora \E $ to $ \A, $ $ g_\gamma $ is a map from $ (\E \tensora \E)_\gamma $ to $ \A_\gamma. $ But we have the isomorphism $ \xi $ from $ \E_\gamma \otimes_{\A_\gamma} \E_\gamma $ to $ (\E \tensora \E)_\gamma $ (Proposition \ref{11thjuly20191}). As in Subsection \ref{18thseptember20192}, we will make an abuse of notation to denote the map $ g_\gamma \xi^{-1} $ by the symbol $ g_\gamma. $ 

\begin{thm} \label{29thaugust20191sm}
	If $g$ is a bi-invariant pseudo-Riemannian metric on a finite bicovariant $\A$-bimodule $\E$ (as in Remark \ref{20thjune}) and $\gamma$ is a 2-cocycle on $\A$, then $g$ deforms to a right $\A_\gamma$-linear map $g_\gamma$ from $\E_\gamma \otimes_{\A_\gamma} \E_\gamma$ to itself. Moreover, $g_\gamma$ is a bi-invariant pseudo-Riemannian metric on $\E_\gamma$. Finally, any bi-invariant pseudo-Riemannian metric on $\E_\gamma$ is a deformation (in the above sense) of some bi-invariant pseudo-Riemannian metric on $\E$.
\end{thm}
\begin{proof} 
	Since $g$ is a right $\A$-linear bicovariant map (Proposition \ref{29thaugust20192}), $g$ indeed deforms to a right $\A_\gamma$-linear map $g_\gamma$ from $(\E \tensora \E)_\gamma \cong \E_\gamma \otimes_{\A_\gamma} \E_\gamma$ (see Proposition \ref{11thjuly20191}) to $\A_\gamma$. The second assertion of Proposition \ref{11thjuly20192} implies that $g_\gamma$ is bicovariant. Then Proposition \ref{29thaugust20192} implies that $g_\gamma$ is bi-invariant. Since $g \circ \sigma = g$, part (iii) of Proposition \ref{11thjuly20192} implies that $$ g_\gamma = (g \circ \sigma)_\gamma = g_\gamma \circ \sigma_\gamma. $$ This verifies condition (i) of Definition \ref{24thmay20191}\\
	Next, we prove that $g_\gamma$ satisfies (ii) of Definition \ref{24thmay20191}. Let $\omega$ be an element of $\zeroE = {}_0(\E_\gamma)$ and $\eta$ be an element of $\Ezero = (\E_\gamma)_0$. Then we have
	\begin{equation*}
	\begin{aligned}
	&(V_g)_\gamma(\omega)(\eta) = (V_g(\omega))_\gamma(\eta) = V_g(\omega)(\eta)\\
	=& g(\omega \tensora \eta) = g_\gamma(\overline{\gamma}(1 \tensorc \eta_{(-1)}) \omega_{(0)} \otimes_{\A_\gamma} \eta_{(0)} \gamma(\omega_{(1)} \tensorc 1))\\
	& {\rm (} {\rm by} ~ {\rm the} ~ {\rm definition} ~ {\rm of} ~ \xi^{-1} ~ {\rm in} ~ {\rm Proposition} \ \ref{11thjuly20191} {\rm)}\\
	=& g_\gamma(\epsilon(\eta_{(-1)}) \omega_{(0)} \otimes_{\A_\gamma} \eta_{(0)} \epsilon(\omega_{(1)})) = g_\gamma(\omega \otimes_{\A_\gamma} \eta) = V_{g_\gamma}(\omega)(\eta).
	\end{aligned}
	\end{equation*}
	Then, by the right-$\A_\gamma$ linearity of $(V_g)_\gamma(\omega)$ and $V_{g_\gamma}(\omega)$, we get, for all $a$ in $\A$,
	\begin{equation*}
	V_{g_\gamma}(\omega)(\eta *_\gamma a) = V_{g_\gamma}(\omega)(\eta)*_\gamma a = (V_g)_\gamma(\omega)(\eta)*_\gamma a = (V_g)_\gamma(\omega)(\eta *_\gamma a).
	\end{equation*}
	Therefore, by the right $\A$-totality of $(\E_\gamma)_0 = \E_0$ in $\E_\gamma$, we conclude that the maps $(V_g)_\gamma$ and $V_{g_\gamma}$ agree on ${}_0(\E_\gamma)$. But since ${}_0(\E_\gamma) = \zeroE$ is right $\A_\gamma$-total in $\E_\gamma$ and both $V_{g_\gamma}$ and $ (V_{g})_\gamma$ are right-$\A_\gamma$ linear, $(V_g)_\gamma = V_{g_\gamma}$ on the whole of $\E_\gamma$.\\
	Next, since $V_g$ is a right $\A$-linear isomorphism from $\E$ to $V_g(\E)$, hence by Corollary \ref{28thaugust20192sm}, $(V_g)_\gamma$ is an isomorphism onto $ (V_g (\E))_\gamma = (V_g)_\gamma (\E_\gamma)=V_{g_\gamma}(\E_\gamma).$ Therefore $V_{g_\gamma}$ is an isomorphism from $\E_\gamma$ to $V_{g_\gamma}(\E_\gamma)$. Hence $g_\gamma$ satisfies (ii) of Definition \ref{24thmay20191}.\\
	To show that every pseudo-Riemannian metric on $\E_\gamma$ is obtained as a deformation of a pseudo-Riemannian metric on $\E$, we view $\E$ as a cocycle deformation of $\E_\gamma$ under the cocycle $\gamma^{-1}$. Then given a pseudo-Riemannian metric $g^\prime$ on $\E_\gamma$, by the first part of this proof, $(g^\prime)_{\gamma^{-1}}$ is a bi-invariant pseudo-Riemannian metric on $\E$. Hence, $g^\prime = ((g^\prime)_{\gamma^{-1}})_{\gamma}$ is indeed a deformation of the bi-invariant pseudo-Riemannian metric $(g^\prime)_{\gamma^{-1}}$ on $\E$.
\end{proof}	

\begin{rmk} \label{26thjune20201}
	We have actually used the fact that $\E$ is finite in order to prove Theorem \ref{29thaugust20191sm}. Indeed, since $\E$ is finite, we can use the results of Section \ref{section3} to derive Proposition \ref{21staugust20191}
	which is then used to prove Corollary  \ref{28thaugust20192sm}. Finally, Corollary \ref{28thaugust20192sm} is used to prove Theorem \ref{29thaugust20191sm}.
	
	Also note that the proof of Theorem \ref{29thaugust20191sm} also implies that the maps $(V_g)_\gamma$ and $V_{g_\gamma}$ are equal.
\end{rmk}

When $g$ is a pseudo-Riemannian bicovariant bilinear metric on $\E,$ then we have a much shorter proof of the fact that $g_\gamma$ is a pseudo-Riemannian metric on $\E_\gamma$ which avoids bicovariant right $\A$-modules. We learnt the proof of this fact from communications with the referee and is as follows:
We will work in the categories $\bimodbicov$ and $\bimodbicovgamma.$ Firstly, as $g$ is bilinear, $V_g$ is a morphism of the category $\bimodbicov.$ and can be deformed to a bicovariant $\A_\gamma$-bilinear map $ ( V_g )_\gamma $ from   $ \E_\gamma  $ to $ (  \E^* )_\gamma. $ Similarly, $g$ deforms to a $\A_\gamma$-bilinear map from $ \E_\gamma \otimes_{\A_\gamma} \E_\gamma $ to $\A_\gamma.$ Then as in the proof of Theorem \ref{29thaugust20191sm}, we can easily check that $ ( V_g )_\gamma = V_{g_\gamma}. $

On the other hand, it is well-known that the left dual $\widetilde{\E}$ of $\E$ is isomorphic to $\E^*.$  Since $g$ is bilinear, Proposition \ref{25thjune202} implies that the morphism $V_g$ (in the category $ \bimodbicov $) is an isomorphism from $\E$ to $\E^*.$

Therefore, we have an isomorphism $ ( V_g )_\gamma $ is an isomorphism from $\E_\gamma$ to $ ( \E^* )_\gamma  \cong  ( \E_\gamma )^* $ by Exercise 2.10.6 of \cite{etingof}.  As $ ( V_g )_\gamma = V_{g_\gamma}, $ we deduce that $ V_{g_\gamma} $ is an isomorphism from $\E_\gamma$ to $ (  \E_\gamma )^*. $ Since the equation $ g_\gamma \circ \sigma_\gamma = g_\gamma, $ this completes the proof.

\vspace{4mm}

{\bf Acknowledgement:} We are immensely grateful to the referee for clarifying the connections with monoidal categories and their deformations. His/her enlightening remarks has greatly improved our understanding. We would like to thank the referee for pointing out the relevant results in \cite{majidcocycle}, \cite{schauenberg} and \cite{bichon}. Proposition 3.5, Remark \ref{26thjune20202}, 
Remark \ref{26thjune20201} and the discussion at the very end of the article are due to the referee.

\end{document}